\documentclass[3p]{elsarticle}

\usepackage{amssymb}

\journal{}
\date{}

\usepackage{amsthm}
\usepackage{mathrsfs}
\usepackage{array}
\usepackage{amsmath}
\usepackage{amsfonts}
\usepackage{dsfont}
\usepackage{latexsym}
\usepackage{graphicx}
\usepackage{epstopdf}
\usepackage{color}

\newtheorem{theorem}{Theorem}[section]
\newtheorem{lemma}[theorem]{Lemma}
\newtheorem{corollary}[theorem]{Corollary}

\newtheorem{conjecture}[theorem]{Conjecture}

\begin{document}

\begin{frontmatter}

\title{Maxima of spectral radius of irregular graphs with given maximum degree}

\author{Jie Xue}\ead{xuejie@zzu.edu.cn}
\author{Ruifang Liu\corref{cor1}}\ead{rfliu@zzu.edu.cn}
\address{School of Mathematics and Statistics, Zhengzhou University, Zhengzhou, Henan 450001, China}
\cortext[cor1]{Corresponding author.}

\begin{abstract}
Let $\lambda^{*}$ be the maximum spectral radius of connected irregular graphs on $n$ vertices with maximum degree $\Delta$.
Liu, Shen and Wang (2007) conjectured that $\lim_{n\rightarrow \infty}(n^{2}(\Delta-\lambda^{*}))/(\Delta-1)=\pi^{2},$
which describes the asymptotic behavior for the maximum spectral radius of irregular graphs.
Focusing on this conjecture, we consider the maximum spectral radius of connected subcubic bipartite graphs.
 The unique connected subcubic bipartite graph with the maximum spectral radius is determined.
Let $G$ be a $k$-connected irregular graph with spectral radius $\lambda_{1}(G)$, we present a lower bound for $\Delta-\lambda_{1}(G)$.
Moreover, if $H$ is a proper subgraph of a $k$-connected $\Delta$-regular graph, a lower bound for $\Delta-\lambda_{1}(H)$ is also obtained.
These bounds improve some previous results.
\end{abstract}


\end{frontmatter}

\section{Introduction}
Let $G$ be a graph with the vertex set $V(G)=\{v_{1},\ldots, v_{n}\}$ and the dege set $E(G)=\{e_{1},\ldots,e_{m}\}$.
The numbers $|V(G)|$ and $|E(G)|$ are called the order and size of $G$, respectively.
For any vertex $v\in V(G)$, the set of the vertices adjacent to $v$ is denoted by $N_{G}(v)$ (or simply $N(v)$).
The number of vertices in $N(v)$, denoted by $d_{G}(v)$, is the degree of $v$.
Let $A(G)=(a_{i,j})$ be the adjacency matrix of $G$, where $a_{i,j}=1$ if $v_{i}$ is adjacent to $v_{j}$, and $a_{i,j}=0$ otherwise.
The eigenvalues of $A(G)$ are denoted by $\lambda_{1}(G)\geq \lambda_{2}(G)\geq \cdots\geq \lambda_{n}(G)$.
As usual, we call $\lambda_{1}(G)$ the spectral radius of $G$.

Let $S$ be a given set of graphs. A graph $G\in S$ is called a maximal graph in $S$
if $\lambda_{1}(G)\geq \lambda_{1}(H)$ for any graph $H\in S$.
Let $\mathcal{F}(n,\Delta)$ denote the set of all connected irregular graphs on $n$ vertices with maximum degree $\Delta$.
Liu, Shen and Wang \cite{Liu2007} proposed the following conjecture concerning the spectral radius of the maximal graph in $\mathcal{F}(n,\Delta)$.

\begin{conjecture}{\rm (Liu, Shen and Wang \cite{Liu2007})}
Let $\lambda^{*}$ be the spectral radius of a maximal graph in $\mathcal{F}(n,\Delta)$. For each fixed $\Delta,$
$$\lim_{n\rightarrow \infty}\frac{n^{2}(\Delta-\lambda^{*})}{\Delta-1}=\pi^{2}.$$
\end{conjecture}

Indeed, the only graph in $\mathcal{F}(n,2)$ is $P_{n}$, and its spectral radius is $2\cos(\frac{\pi}{n+1})$.
It is obvious that the conjecture holds for $\Delta=2$.
This conjecture describes the asymptotic behavior for the maximum spectral radius of irregular graphs.
To determine the asymptotic value for the maximum spectral radius, the key is to find the maximal irregular graph.
Inspired by the conjecture, we consider the maximum spectral radius of connected subcubic bipartite graphs.
Obviously, there are at least four vertices in a connected subcubic bipartite graph, and $K_{1,3}$ is the only one connected subcubic bipartite graph on four vertices.
For connected subcubic bipartite graphs with five vertices, it is easy to see that $K_{2,3}$ is the unique maximal graph.
For the remaining cases, we show that $B_{n}$ (defined in Section 2) is the unique maximal graph among all connected subcubic bipartite graphs.

\begin{theorem}\label{main3}
If $G$ is a connected subcubic bipartite graph on $n\geq 6$ vertices,
then $\lambda_{1}(G)\leq \lambda_{1}(B_{n})$, with equality if and only if $G\cong B_{n}$.
\end{theorem}

Given a connected graph $G$ with maximum degree $\Delta$.
It is well-known that $\lambda_{1}(G)=\Delta$ if $G$ is $\Delta$-regular, and $\lambda_{1}(G)<\Delta$ if $G$ is irregular.
It is natural to ask how small $\Delta-\lambda_{1}(G)$ can be when $G$ is irregular.
A lower bound for $\Delta-\lambda_{1}(G)$ was first given by Stevanovi\'{c}.

\begin{theorem}{\rm (Stevanovi\'{c} \cite{Stevanovic2004})}\label{th1}
  Let $G$ be a connected irregular graph of order $n$ and maximum degree $\Delta$. Then
  \begin{equation}\label{eq18}
    \Delta-\lambda_{1}(G)>\frac{1}{2n(n\Delta-1)\Delta^{2}}.
  \end{equation}
\end{theorem}
Some lower bounds for $\Delta-\lambda_{1}(G)$ under other graph parameters,
such as the diameter and the minimum degree, were presented in \cite{Cioaba2007EJC,Cioaba2007JCTB,Liu2007,Shi2009,Zhang2005}.
In particular, Cioab\u{a} \cite{Cioaba2007EJC} proved the following lower bound, which solved a conjecture in \cite{Cioaba2007JCTB}.
\begin{theorem}{\rm (Cioab\u{a} \cite{Cioaba2007EJC})}
  Let $G$ be a connected irregular graph with $n$ vertices, maximum degree $\Delta$ and diameter $D$. Then
  \begin{equation}\label{eq22}
    \Delta-\lambda_{1}(G)>\frac{1}{nD}.
  \end{equation}
\end{theorem}
Recently, Chen and Hou obtained a lower bound for $\Delta-\lambda_{1}(G)$ in terms of the connectivity of $G$.
\begin{theorem}{\rm (Chen and Hou \cite{Chen2014})}\label{th2}
  Let $G$ be a $k$-connected irregular graph of order $n$ and size $m$. If the maximum degree of $G$ is $\Delta$, then
  \begin{equation}\label{eq19}
    \Delta-\lambda_{1}(G)>\frac{(n\Delta-2m)k^{2}}{(n\Delta-2m)\left(n^{2}-2n+2k\right)+nk^{2}}.
  \end{equation}
\end{theorem}

Inspired by these results, we consider the spectral radius of $k$-connected irregular graphs.
The following result improves the bounds in Theorems \ref{th1} and \ref{th2}.

\begin{theorem}\label{main1}
  Let $G$ be a $k$-connected irregular graph of order $n$ and size $m$. If the maximum degree of $G$ is $\Delta$, then
  \begin{equation}\label{eq0}
    \Delta-\lambda_{1}(G)>\frac{(n\Delta-2m)k^{2}}{(n\Delta-2m)\left((n-1)^{2}-(n-k-1)(\Delta-k+1)\right)+nk^{2}}.
  \end{equation}
\end{theorem}

As mentioned above, the spectral radius of a $\Delta$-regular graph is equal to $\Delta$.
It is clear that the spectral radius of a proper subgraph of a connected $\Delta$-regular graph is smaller than $\Delta$.
Let $H$ be a proper subgraph of a connected $\Delta$-regular graph with the spectral radius $\lambda_{1}(H)$.
Cioab\v{a} \cite{Cioaba2007EJC} presented a lower bound for $\Delta-\lambda_{1}(H)$ when $H$ is obtained from the $\Delta$-regular supergraph by deleting an edge.
Later, Nikiforov \cite{Nikiforov2007} proved a similar bound for $\Delta-\lambda_{1}(H)$ if $H$ is an arbitrary proper subgraph of the $\Delta$-regular supergraph.
In \cite{Shi2009}, Shi improved the bounds of Cioab\v{a} and Nikiforov.
Recently, Chen and Hou \cite{Chen2014} studied the problem under the condition of connectivity.
Let $H$ be a proper subgraph of a connected $\Delta$-regular $k$-connected graph.
The following lower bound for $\Delta-\lambda_{1}(H)$ was obtained by Chen and Hou.

\begin{theorem}{\rm (Chen and Hou \cite{Chen2014})}\label{th3}
  Let $H$ be a proper subgraph of a $k$-connected $\Delta$-regular graph $G$ with $n$ vertices. If $k\geq 2$, then
  \begin{equation}\label{eq21}
    \Delta-\lambda_{1}(H)>\frac{(k-1)^{2}}{(n-\Delta)(n-\Delta+2k-4)+n(k-1)^{2}}.
  \end{equation}
\end{theorem}
In this paper we delete the requirement $k\geq2$ in Theorem \ref{th3}, and present an improved lower bound for $\Delta-\lambda_{1}(H).$
\begin{theorem}\label{main2}
  Let $G$ be a $k$-connected $\Delta$-regular graph of order $n$. If $H$ is a proper subgraph of $G$, then
  \begin{equation}\label{eq20}
    \Delta-\lambda_{1}(H)>\frac{k^{2}}{(n-\Delta-1)(n-\Delta+2k-2)+nk^{2}}.
  \end{equation}
\end{theorem}

The rest of the paper is organized as follows. In the next section, we present the proof of Theorem \ref{main3}.
The proofs of Theorems \ref{main1} and \ref{main2} are given in Section 3. Finally, some discussions are presented.

\section{Spectral radius of subcubic bipartite graphs}

Eigenvector is an effective tool to study the spectral radius.
Let $G$ be a graph with spectral radius $\lambda_{1}(G)$. By the Rayleigh quotient,
$$\lambda_{1}(G)=\max\frac{y^{t}A(G)y}{y^{t}y}=\max_{||y||=1}y^{t}A(G)y.$$
Note that $A(G)$ is a nonnegative matrix. Then, by Perron-Frobenius theorem (see, e.g., \cite[Theorem 4.2]{Minc1988}),
there exists a nonnegative eigenvector corresponding to $\lambda_{1}(G)$. Furthermore, if $G$ is connected, then $A(G)$ is irreducible,
and there is a positive eigenvector corresponding to $\lambda_{1}(G).$
Suppose that $x$ is a unit eigenvector of $A(G)$ corresponding to $\lambda_{1}(G)$, and such a vector is called the principal eigenvector of $G$.
The vector $x$ can be considered as a function on the vertices of $G$, where $x_{v}$ means the entry of $x$ corresponding to $v\in V(G)$. Thus,
$$\lambda_{1}(G)=x^{t}A(G)x=\sum_{uv\in E(G)}2x_{u}x_{v}.$$
Let us recall some useful results. The following result presents an operation,
which increasing the spectral radius without changing the degree sequence.

\begin{lemma}{\rm (\cite{Cvetkovic1997})}\label{trans}
Let $u,u',v,v'$ be four distinct vertices of a connected graph $G$ and let $uv',u'v\in E(G)$, while $uv,u'v'\notin E(G)$.
Let $x$ be the principal eigenvector of $G.$ If $x_{u}\geq x_{u'}$ and $x_{v}\geq x_{v'}$, then
  $$\lambda_{1}(G+uv+u'v'-uv'-u'v)\geq \lambda_{1}(G),$$
  with equality if and only if $x_{u}= x_{u'}$ and $x_{v}=x_{v'}$.
\end{lemma}

Two edges are independent if they have no end-vertex in common.
For a connected graph $G$, let $uv'$ and $u'v$ be two independent edges in $G$.
We say that $\{uv', u'v\}$ is a bad pair of edges in $G$ if it satisfies the following conditions:
\begin{itemize}
  \item $u\not\sim v$ and $u'\not\sim v'$;
  \item $x_{u}\geq x_{u'}$ and $x_{v}>x_{v'}$;
  \item $G+\{uv, u'v'\}-\{uv',u'v\}$ is connected.
\end{itemize}
According to Lemma \ref{trans}, one can see that the spectral radius of $G+\{uv, u'v'\}-\{uv',u'v\}$ is greater than that of $G$.
Thus, we obtain the following consequence directly.

\begin{corollary}\label{nobadedge}
  Let $\mathcal{F}$ be the set of all connected graphs with a given degree sequence.
  If $G$ is a maximal graph in $\mathcal{F}$, then there is no bad pair of edges in $G$.
\end{corollary}

The next lemma gives another operation on the neighbours of two distinct vertices, which also enables the increase of the spectral radius.

\begin{lemma}{\rm (\cite{Cvetkovic1997})}\label{lem-cve}
  Let $u,v$ be two vertices of a connected graph $G$ and let $S\subseteq N(u)\backslash N(v)$.
Let $x$ be the principal eigenvector of $G.$ Define
 $$G'=G-\{wu:w\in S\}+\{wv:w\in S\}.$$
If $S\neq \emptyset$ and $x_v\geq x_u$, then $\lambda_{1}(G')>\lambda_{1}(G)$.
\end{lemma}

Let $\mathcal{B}(n, \Delta)$ be the set of all connected irregular bipartite graphs on $n$ vertices with maximum degree $\Delta$.
Clearly,  $\mathcal{B}(n, 3)$ means the set of all connected subcubic bipartite graphs on $n$ vertices.
Lemma \ref{lem-cve} implies the following property of the principal eigenvector for the maximal graph in $\mathcal{B}(n, \Delta)$.

\begin{lemma}\label{eige-ord}
  Let $G$ be a maximal graph in $\mathcal{B}(n, \Delta).$ Let $u$ and $v$ be two vertices in the same part of $G$.
  If $d_{G}(u)>d_{G}(v)$, then $x_u>x_v$, where $x$ is the principal eigenvector of $G$.
\end{lemma}
\begin{proof}
 Suppose to the contrary that $x_u\leq x_v.$ Let $S=N(u)\backslash N(v)$.
 Clearly, $S\neq \emptyset$ and there exists a vertex $w\in S$ such that the graph $G'=G-wu+wv$ is also connected.
 We claim that the other part of $G$ contains a vertex with degree $\Delta$.
 If not, one can see that $v$ is adjacent to all vertices of the other part. However, this contradicts the fact $d_{G}(u)>d_{G}(v)$.
 It follows that the graph $G'$ also belongs to $\mathcal{B}(n,\Delta)$.
   Then Lemma \ref{lem-cve} shows that $\lambda_{1}(G')>\lambda_{1}(G)$, contradicting the maximality of $G$.
\end{proof}

The spectral radius of the maximal graph in $\mathcal{B}(n, \Delta)$ becomes large as $n$ increases.

\begin{lemma}\label{max-ord}
  Let $G$ and $H$ be maximal graphs in $\mathcal{B}(n, \Delta)$ and $\mathcal{B}(n', \Delta)$, respectively.
  If $n>n'$, then $\lambda_{1}(G)>\lambda_{1}(H)$.
\end{lemma}
\begin{proof}
  It suffices to prove the lemma for $n'=n-1$. Since $H$ is irregular, there exists a vertex $v\in H$ with degree less than $\Delta$.
  Add a new vertex $u$ and edge $uv$, clearly the resulting graph $H+uv$ belongs to $\mathcal{B}(n,\Delta)$.
  It follows that $\lambda_{1}(G)\geq \lambda_{1}(H+uv)>\lambda_{1}(H)$, which completes the proof.
\end{proof}

Suppose that, in a bipartite graph, the degree sequences of vertices belonging to distinct parts
are $(a_{1},\ldots,a_{n_{1}})$ and $(b_{1},\ldots,b_{n_{2}})$, respectively.
Then we write $(a_{1},\ldots,a_{n_{1}}\mid b_{1},\ldots,b_{n_{2}})$ to denote the degree sequence of the bipartite graph.
To prove Theorem \ref{main3}, we need to determine the degree sequence of the maximal graph in $\mathcal{B}(n,3)$.
In the maximal graph, if the degree of a vertex is less than 3, we call it an unsaturated vertex.

\begin{lemma}\label{lem-degree}
  If $G$ is a maximal graph in $\mathcal{B}(n,3)$ with $n\geq 6$, then the degree sequence of $G$ is either
  $$(\underbrace{3,3,\ldots,3,2}_{n/2}\mid\underbrace{3,3,\ldots,3,2}_{n/2})$$
  or
  $$(\underbrace{3,3,\ldots,3}_{(n-1)/2}\mid\underbrace{3,3,\ldots,3,2,1}_{(n+1)/2}).$$
\end{lemma}
\begin{proof}
  Assume that the bipartition of $G$ is $(X,Y)$. Define vertex sets $X^{*}=\{v\in X:d_{G}(v)<3\}$ and $Y^{*}=\{v\in Y:d_{G}(v)<3\}$.
  The vertices in $X^{*}$ and $Y^{*}$ are all unsaturated.
  In order to prove the lemma, we need the following facts.

\medskip

  \noindent\textbf{Fact 1}. If $u\in X^{*}$ and $v\in Y^{*}$ are nonadjacent, then $d_{G}(u)=d_{G}(v)=2$ and $(X^{*}\cup Y^{*})\backslash\{u,v\}=\emptyset$.

  \begin{proof}[Proof of Fact 1]
  If $d_{G}(u)\neq 2$ or $d_{G}(v)\neq 2$, then the graph $G+uv$ also belongs to $\mathcal{B}(n,3)$.
  However, we can see that $\lambda_{1}(G+uv)>\lambda_{1}(G),$ contradicting the maximality of $G$. This implies that $d_{G}(u)=d_{G}(v)=2$.
  If $(X^{*}\cup Y^{*})\backslash\{u,v\}\neq\emptyset$, then $G+uv\in \mathcal{B}(n,3)$ and $\lambda_{1}(G+uv)>\lambda_{1}(G)$,
  which also contradicts the maximality of $G.$
  \end{proof}

  \noindent\textbf{Fact 2}. Let $u$ and $v$ be two unsaturated vertices in the same part. If $d_{G}(u)=d_{G}(v)=1$, then $N(u)=N(v)$.

  \begin{proof}[Proof of Fact 2]
  Without loss of generality, assume that $x_u\geq x_v$. Set $N(v)=\{w\}$.
  If $N(u)\neq N(v)$, then $u$ is not adjacent to $w$. Let $G'=G-vw+uw$. By Lemma \ref{lem-cve}, $\lambda_{1}(G')>\lambda_{1}(G)$.
  Note that $v$ is an isolated vertex in $G'$, and $G'-v$ belongs to $\mathcal{B}(n-1,3)$. Thus
$$\lambda_{1}(G'-v)=\lambda_{1}(G')>\lambda_{1}(G),$$
contradicting Lemma \ref{max-ord}.
  \end{proof}

   \noindent\textbf{Fact 3}. Let $u$ and $v$ be two unsaturated vertices in the same part. If $d_{G}(u)=d_{G}(v)=2$, then $N(u)=N(v)$.

   \begin{proof}[Proof of Fact 3]
    By contradiction, suppose that $N(u)\neq N(v)$. Without loss of generality, we assume that $x_u\geq x_v$.
  Note that $|N(v)\backslash N(u)|=1$ or 2. It is obvious that there is a vertex in $N(v)\backslash N(u)$, say $w$,
  such that the graph $G'=G-vw+uw$ is connected.
  Then the graph $G'$ also belongs to $\mathcal{B}(n,3)$.
  It follows from Lemma \ref{lem-cve} that $\lambda_{1}(G')>\lambda_{1}(G)$, contradicting the maximality of $G$.
   \end{proof}

  \noindent\textbf{Fact 4}. $|X^{*}|\leq 3$ and $|Y^{*}|\leq 3$.

  \begin{proof}[Proof of Fact 4]
  If $X^{*}$ has three vertices with degree 1,
  then Fact 2 implies that these three pendant vertices are adjacent to a common vertex in $Y$.
  Note also that the order of $G$ is at least 6.
  This leads to that $G$ is disconnected, a contradiction. Hence, $X^{*}$ contains at most two vertices of degree 1.
  Similarly, by Fact 3, $X^{*}$ contains at most two vertices of degree 2.

  If $|X^{*}|\geq 4$, then $X^{*}$ contains exactly two vertices of degree 1 and two vertices of degree 2, i.e., $|X^{*}|=4$.
  Then, we choose two vertices $u,v\in X^{*}$ with $d_{G}(u)=2$ and $d_{G}(v)=1$.
  Set $N(v)=\{w\}$. One can see that $w\not\sim u$ (if not, Facts 2 and 3 yield that $d_{G}(w)=4$, a contradiction).
  By Lemma \ref{eige-ord}, $x_{u}>x_{v}$.
  Let $G'=G+uw-vw$. It follows from Lemma \ref{lem-cve} that $\lambda_{1}(G')>\lambda_{1}(G)$.
  Note that $v$ is an isolated vertex in $G'$, and $G'-v\in \mathcal{B}(n-1,3)$.
  According to Lemma \ref{max-ord}, we have $\lambda_{1}(G)>\lambda_{1}(G'-v)=\lambda_{1}(G')$, a contradiction. Therefore, $|X^{*}|\leq 3$.
  A similar argument can assure that $|Y^{*}|\leq 3$.
  \end{proof}

  In the following, we consider the vertices in $X^{*}\cup Y^{*}$. It can be divided into two cases according to whether $X^{*}\cup Y^{*}$ has some vertices with degree 1.

  If there exists a vertex in $X^{*}\cup Y^{*}$ with degree 1, then we assume that $d_{G}(u)=1$ where $u\in X^{*}$.
  We claim that $Y^{*}=\emptyset$. Suppose not. Then it follows from Fact 1 that $|Y^{*}|=1$. Set $Y^{*}=\{v\}$. Obviously, $d_{G}(v)=2$ and $v\sim u$.
  Again by Fact 1, $v$ is adjacent to all vertices of $X^{*}$, and so $|X^{*}|\leq 2$. If $|X^{*}|=1$, then, by counting the number of edges, we have
  $$3(|X|-1)+1=3(|Y|-1)+2.$$
  This implies that $|X|-|Y|=1/3$, a contradiction. If $|X^{*}|=2$, then we assume $X^{*}=\{u,w\}$.
  Since $G$ is connected, we have $d_{G}(w)=2$.
  By counting the number of edges, we have
  $$3(|X|-2)+3=3(|Y|-1)+2,$$
  which yields that $|X|-|Y|=2/3$, a contradiction. Therefore, $Y^{*}=\emptyset$. Let us consider the number of vertices in $X^{*}$.
  Recall that $u\in X^{*}$ and $d_{G}(u)=1$. Suppose that the degree sum of vertices in $X^{*}\backslash \{u\}$ is equal to $a$.
  Since $Y^{*}=\emptyset$, the degree sum of vertices in $X^{*}$ is a multiple of 3, that is, $3|(a+1)$.
  Since $|X^{*}|\leq 3$, it follows that $a=2$.
  As mentioned in the proof of Fact 4, $X^{*}$ cannot contain three vertices of degree 1.
  Thus, $|X^{*}|=2$, and the degrees of vertices in $X^{*}$ are $2$ and $1$.
  By counting the number of edges, we have
  $$3(|X|-2)+a+1=3|Y|,$$ which leads to $|X|=|Y|+1$. Therefore, the degree sequence of $G$ is $(3,\ldots,3,2,1\mid 3,\ldots,3)$, as required.

   Now, suppose that the degree of any vertex in $X^{*}\cup Y^{*}$ is 2.
   We claim that $X^{*}$ and $Y^{*}$ are both nonempty. Assume that $|Y^{*}|=0$. The degree sum of vertices in $X$ is a multiple of 3.
   This implies that $|X^{*}|=3$. According to Fact 3, one can see that $K_{2,3}$ is a component of $G$, which contradicts the connectivity of $G$.
   Therefore, $3\geq |Y^{*}|\geq 1$. Similarly, we have $3\geq |X^{*}|\geq 1$.

   We may assume that $|X^{*}|\geq |Y^{*}|$.
   If $|Y^{*}|\geq 2$, then $|X^{*}|\geq 2$. Choose vertices $u_{1},u_{2}\in X^{*}$ and $v_{1},v_{2}\in Y^{*}$.
    By Fact 1, it is obviously that $u_{i}$ is adjacent to $v_{j}$ for any $1\leq i,j\leq 2$.
    Since $d_{G}(u_{1})=d_{G}(u_{2})=d_{G}(v_{1})=d_{G}(v_{2})=2$, these four vertices form a component of $G$, which contradicts the connectivity of $G$.
    Thus, $|Y^{*}|=1$. If $|X^{*}|=3$, then it follows from Fact 3 that $G$ contains a component $K_{2,3}$, a contradiction.
    Hence $|X^{*}|\leq 2$.
    If $|X^{*}|=2$, by counting the number of edges, then we have
    $$3(|X|-2)+4=3(|Y|-1)+2.$$
  This implies that $|X|-|Y|=1/3$, a contradiction. Therefore, $|X^{*}|=1$, and so the degree sequence of $G$ is $(3,\ldots,3,2\mid 3,\ldots,3,2)$.
\end{proof}

All possible degree sequences of the maximal graph have been determined.
One can see that, in the maximal graph, there are exactly two unsaturated vertices.
The next lemma presents more properties for the maximal graph.

\begin{lemma}\label{lem-nocut}
  Let $G$ be a maximal graph in $\mathcal{B}(n,3)$ with $n\geq 6$.\\
  \noindent(1) If $u$ is a vertex in $G$ with degree 2, then $u$ cannot be incident with a cut edge.\\
  \noindent(2) If $G$ contains a cut edge, and $v,v'$ are the only two unsaturated vertices, then any path from $v$ to $v'$ passes through the cut edge.
\end{lemma}
\begin{proof}
  (1) By contradiction, assume that $u$ is incident with a cut edge $e$ in $G$. Let $H_{1}$ and $H_{2}$ be the two components in $G-e$, where $u\in V(H_{1})$.
  Clearly, $H_{1}$ and $H_{2}$ are both bipartite graphs.
  Note that $(3,\ldots,3,2)$ and $(3,\ldots,3,1)$ cannot be the degree sequence of a bipartite graph.
  According to Lemma \ref{lem-degree}, we can see that, in $G$, there is another vertex $v$ with degree $d_{G}(v)<3$.
  If $v\in V(H_{1})$, then the degree sequence of the induced subgraph $H_{2}$ is $(3,3,\ldots,3,2)$, a contradiction.
  On the other hand, if $v\in V(H_{2})$, then the degree sequence of the induced subgraph $H_{1}$ is $(3,3,\ldots,3,1)$, a contradiction.
  Then statement (1) holds.

  (2) Suppose that $e$ is a cut edge in $G$. Let $H_{1}$ and $H_{2}$ be the two components in $G-e$.
  If $v$ and $v'$ belong to the same component, say $H_{1}$, then $d_{G}(w)=3$ for each vertex $w\in V(H_{2})$,
  hence the degree sequence of the induced subgraph $H_{2}$ is $(3,3,\ldots,3,2)$, a contradiction.
  Thus, $v$ and $v'$ belong to the distinct components, and so the path from $v$ to $v'$ must passes through $e$. The statement (2) is proved.
\end{proof}

Form Lemma \ref{lem-nocut} (2), we obtain the following result directly.

\begin{corollary}\label{nocut}
  Let $G$ be a maximal graph in $\mathcal{B}(n,3)$ with $n\geq 6$. Suppose that $u$ and $v$ are the only two unsaturated vertices.
  If $u$ is adjacent to $v$, then, except for the edge $uv$, there is no cut edge in $G$.
\end{corollary}

\begin{figure}[t]
\begin{center}
{\includegraphics{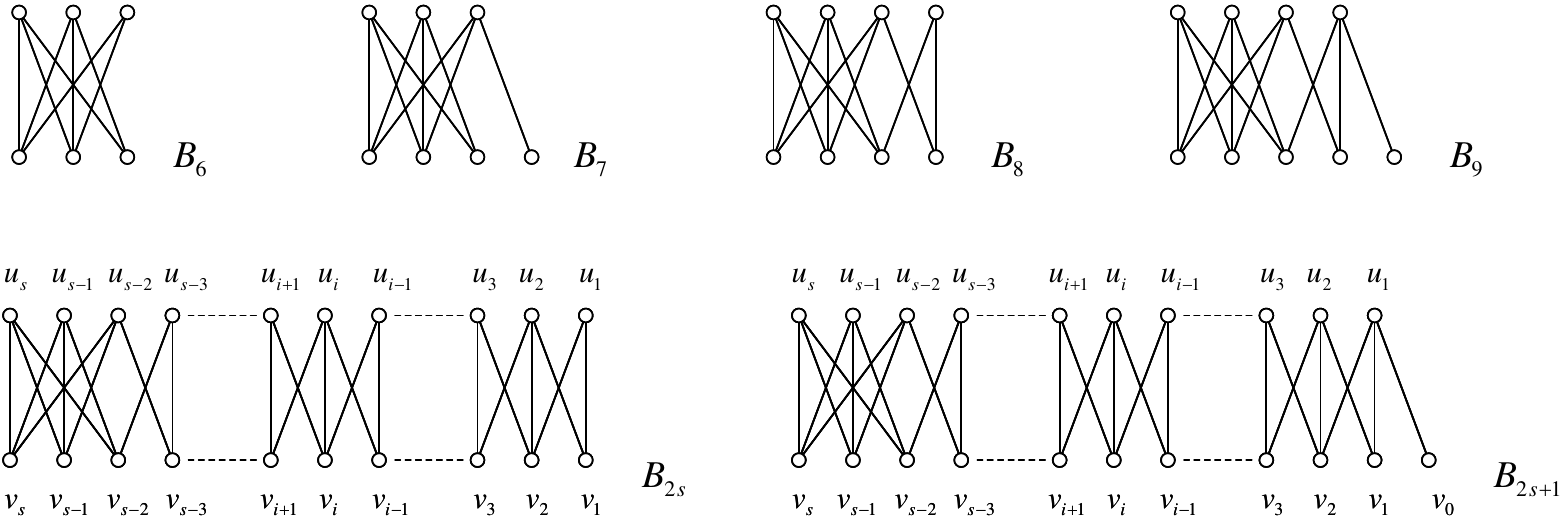}}
\end{center}
\caption{Subcubic bipartite graphs.}
\label{fig1}
\end{figure}

For $n\geq 6$, let $B_{n}$ be a connected subcubic bipartite graph defined as follows:
\begin{itemize}
  \item $B_{6}$ is obtained from $K_{3,3}$ by deleting an edge.
  \item When $n\geq 7$ and $n$ is odd, $B_{n}$ is obtained from $B_{n-1}$ by adding a new vertex and joining it to one of the unsaturated vertices in $B_{n-1}$.
  \item When $n\geq 8$ and $n$ is even, $B_{n}$ is obtained from $B_{n-1}$ by adding a new vertex and joining it to all unsaturated vertices in $B_{n-1}$.
\end{itemize}
For example, subcubic bipartite graphs $B_{6},B_{7},B_{8},B_{9}$ are presented in Figure \ref{fig1}. We are now ready to prove the maximal graph.

\begin{proof}[\bf Proof of Theorem \ref{main3}]
  Let $G$ be a maximal graph in $\mathcal{B}(n,3)$. Lemma \ref{lem-degree} shows that the degree sequence of $G$ is either
  $(3,3,\ldots,3,2\mid 3,3,\ldots,3,2)$ or $(3,3,\ldots,3\mid 3,3,\ldots,3,2,1)$. Suppose that $x$ is a principal eigenvector of $G$.
 We establish the structure of $G$ by analyzing the eigenvector in each case.

 \medskip

\noindent\textbf{Case 1.} The degree sequence of $G$ is $(3,3,\ldots,3,2\mid 3,3,\ldots,3,2)$.

\medskip

If $n=6$, then it is easy to see that $G\cong B_{6}$. In the following, we assume that $n>6$.
Suppose that $G$ has the bipartition $(U,V)$, where $U=\{u_{1},u_{2},\ldots,u_{s}\}$ and $V=\{v_{1},v_{2},\ldots,v_{s}\}$. Without loss of generality, we assume that
$$x_{u_{1}}\leq x_{u_{2}}\leq\cdots\leq x_{u_{s}}~~~~\text{and}~~~~x_{v_{1}}\leq x_{v_{2}}\leq\cdots\leq x_{v_{s}}.$$

\noindent\textbf{Claim 1.} For $1\leq i\leq s-3$, we obtain that $u_{i}\sim v_{i}$, $u_{i}\sim v_{i+1}$, $v_{i}\sim u_{i+1}$ and the eigenvector $x$ satisfies that
$$x_{u_{i}}<x_{u_{i+1}}<x_{u_{i+2}}\leq \cdots\leq x_{u_{s}}~~~~\text{and}~~~~x_{v_{i}}<x_{v_{i+1}}<x_{v_{i+2}}\leq\cdots\leq x_{v_{s}}.$$

\begin{proof}[Proof of Claim 1]

We prove the claim by induction on $i$. By Lemma \ref{eige-ord}, it follows that $d_{G}(u_{1})=d_{G}(v_{1})=2$,
$$x_{u_{1}}<\min\{x_{u_{i}}:2\leq i\leq s\}~~~~\text{and}~~~~x_{v_{1}}<\min\{x_{v_{i}}:2\leq i\leq s\}.$$

We now show that $u_{1}\sim v_{1}$. By contradiction, suppose that $u_{1}\not \sim v_{1}$.
Since $n>6$, we can find two nonadjacent vertices $v\in N(u_{1})$ and $u\in N(v_{1})$.
Let $\mathsf{P}$ be a shortest path from $u_{1}$ to $u$ in $G$.
If the path $\mathsf{P}$ passes through vertices $v_{1}$ and $v$ simultaneously,
then the graph $G+\{u_{1}v_{1},uv\}-\{u_{1}v,uv_{1}\}$ is obviously connected.
Moreover, the graph $G+\{u_{1}v_{1},uv\}-\{u_{1}v,uv_{1}\}$ is also connected,
when $v_{1}$ and $v$ are both not appeared in $\mathsf{P}$.
Since $x_{u_{1}}<x_{u}$ and $x_{v_{1}}<x_{v}$, $\{u_{1}v,uv_{1}\}$ is a bad pair of edges in $G$,
contradicting Corollary \ref{nobadedge}.
Suppose that exactly one of vertices $v_{1}$ and $v$ belongs to $\mathsf{P}$.
Set $N(u_{1})=\{v,v'\}$. If $v_{1}$ belongs to $\mathsf{P}$, then $\mathsf{P}=u_{1}v'\mathsf{P}v_{1}u$.
Thus, $G+\{u_{1}v_{1},uv'\}-\{u_{1}v',uv_{1}\}$ is connected
since the vertices of $\{u_{1},v_{1},u,v'\}$ are connected by a path $u_{1}v_{1}\mathsf{P}v'u$ in this graph.
On the other hand, if $v$ belongs to $\mathsf{P}$, then $\mathsf{P}=u_{1}v\mathsf{P}u$.
Since $v_{1}u_{1}v\mathsf{P}uv'$ is a path in $G+\{u_{1}v_{1},uv'\}-\{u_{1}v',uv_{1}\}$, this graph is connected.
Note also that $x_{u_{1}}<x_{u}$ and $x_{v_{1}}<x_{v'}$.
Thus, $\{u_{1}v',uv_{1}\}$ is a bad pair in $G$, which contradicting Corollary \ref{nobadedge}.
Therefore, we obtain that $u_{1}\sim v_{1}$.

Let $N(u_{1})=\{v_{1},v\}$. Clearly, $v\in \{v_{2},\ldots,v_{s}\}$.
We claim that $x_{v}<\min\{x_{w}:w\in \{v_{2},\ldots, v_{s}\}\}$.
Suppose not, and let $v^{*}\in \{v_{2},\ldots,v_{s}\}\backslash\{v\}$ with $x_{v^{*}}\leq x_{v}$.
Since $d_{G}(v^{*})=d_{G}(v)=3$ and $u_{1}\not\sim v^{*}$, it is obvious that $N(v^{*})\backslash N(v)\neq \emptyset$.
Let $u\in N(v^{*})\backslash N(v)$. Then $u\sim v^{*}$ and $u\not\sim v$.
Let $\mathsf{P}$ be a shortest path from $v$ to $v^{*}$ in $G$.

If the vertices $u_{1}$ and $u$ both belong to (or both not belong to) the path $\mathsf{P}$,
then it is easy to see that the graph $G+\{u_{1}v^{*},uv\}-\{u_{1}v,uv^{*}\}$ is connected.
Since $x_{u}>x_{u_{1}}$ and $x_{v}\geq x_{v^{*}}$, $\{u_{1}v,uv^{*}\}$ forms a bad pair of edges in $G$, contradicting Corollary \ref{nobadedge}.

Suppose that exactly one of vertices $u_{1}$ and $u$ belongs to the path $\mathsf{P}$.
In this case, one can see that $N(v^{*})\cap N(v)=\emptyset$.
If $u_{1}$ belongs to the path $\mathsf{P}$,
then $\mathsf{P}$ contains a vertex $u^{*}\in N(v^{*})\backslash \{u\}$, where $u^{*}\not\sim v$ and $x_{u^{*}}>x_{u_{1}}$.
Hence, $G+\{u^{*}v,u_{1}v^{*}\}-\{u^{*}v^{*}, u_{1}v\}$ is connected,
which yielding that $\{u^{*}v^{*}, u_{1}v\}$ is a bad pair of edges in $G$, a contradiction.
Suppose that $u$ belongs to the path $\mathsf{P}$. Let us choose a vertex $u^{**}\in N(v^{*})\backslash \{u\}$.
Thus, $u^{**}v\mathsf{P}uv^{*}u_{1}$ is a path in $G+\{u^{**}v,u_{1}v^{*}\}-\{u^{**}v^{*}, u_{1}v\}$, and so this graph is connected.
Moreover, since $x_{u^{**}}>x_{u_{1}}$ and $x_{v}\geq x_{v^{*}}$, $\{u^{**}v^{*}, u_{1}v\}$ is a bad pair of edges in $G$, a contradiction.

In above, we both obtain contradictions, hence $x_{v}<\min\{x_{w}:w\in \{v_{2},\ldots, v_{s}\}\}$.
That is $u_{1}\sim v_{2}$ and $x_{v_{2}}<x_{v_{3}}\leq \cdots\leq x_{v_{s}}$.
A similar argument shows that $v_{1}\sim u_{2}$,
and $x_{u_{2}}<x_{u_{3}}\leq \cdots\leq x_{u_{s}}$.
Therefore, Claim 1 holds for $i=1$.

Assume that Claim 1 holds for $i<k$. Now, let us consider the case $i=k$. Thus, we have
$$x_{u_{1}}<x_{u_{2}}<\cdots<x_{u_{k-1}}<x_{u_{k}}<x_{u_{k+1}}\leq\cdots\leq x_{u_{s}}$$
and
$$x_{v_{1}}<x_{v_{2}}<\cdots<x_{v_{k-1}}<x_{v_{k}}<x_{v_{k+1}}\leq\cdots\leq x_{v_{s}}.$$
Moreover, the neighbours of any vertex in $\{u_{1},\ldots,u_{k-1},v_{1},\ldots, v_{k-1}\}$ are determined.
In particular, we also obtain that $u_{k}\sim v_{k-1}$ and $v_{k}\sim u_{k-1}$.
The remaining two neighbours of $u_{k}$ (resp. $v_{k}$) belong to $\{v_{i}: i\geq k\}$ (resp. $\{u_{i}:i\geq k\}$).

To prove $u_{k}\sim v_{k}$. By contradiction, we may assume that $u_{k}\not\sim v_{k}$.
Clearly, $u_{k}$ has two neighbours in $\{v_{k+1},\ldots,v_{s}\}$,
and $v_{k}$ has two neighbours in $\{u_{k+1},\ldots,u_{s}\}$.
Since $k\leq s-3$, we can find two vertices $v\in N(u_{k})\cap \{v_{k+1},\ldots,v_{s}\}$ and $u\in N(v_{k})\cap \{u_{k+1},\ldots,u_{s}\}$ such that $u\not\sim v$.
Assume that $N(u_{k})=\{v_{k-1},v,v'\}$.
Note that $x_{u_{k}}<x_{u}$ and $x_{v_{k}}<x_{v}$.
If $G-\{u_{k}v,v_{k}u\}$ is connected, then $G+\{u_{k}v_{k},uv\}-\{u_{k}v,v_{k}u\}$ is also connected,
and so $\{u_{k}v,v_{k}u\}$ is a bad pair of edges in $G$, a contradiction.
So, we assume that $G-\{u_{k}v,v_{k}u\}$ is disconnected.
By Corollary \ref{nocut}, it is easy to see that $G-\{u_{k}v,v_{k}u\}$ contains two components.
Suppose that $H_{1}$ and $H_{2}$ are the two components in $G-\{u_{k}v,v_{k}u\}$,
where $\{u_{k},v_{k},v'\}\subset H_{1}$ and $\{u,v\}\subset H_{2}$.
Let $\mathsf{P}$ be a path from $u$ to $v$ in $H_{2}$.
Consider the graph $G+\{u_{k}v_{k},uv'\}-\{u_{k}v',uv_{k}\}$.
Since $v_{k}u_{k}v\mathsf{P}uv'$ is a path in the graph $G+\{u_{k}v_{k},uv'\}-\{u_{k}v',uv_{k}\}$,
this graph is connected. Moreover, since $x_{v'}>x_{v_{k}}$ and $x_{u}>x_{u_{k}}$,
it follows that $\{u_{k}v',uv_{k}\}$ is a bad pair edges of $G$, a contradiction.
Therefore, we obtain that $u_{k}\sim v_{k}$.

Suppose, now, that $N(u_{k})=\{v_{k-1},v_{k},v\}$ where $v\in \{v_{k+1},\ldots,v_{s}\}$.
We will show that $$x_{v}<\min\{x_{w}:w\in \{v_{k+1},\ldots,v_{s}\}\backslash\{v\}\}.$$
Suppose not, and let $w\in \{v_{k+1},\ldots,v_{s}\}\backslash\{v\}$ with $x_{w}\leq x_{v}$.
Since $d_{G}(w)=d_{G}(v)=3$, it is obvious that $N(w)\backslash N(v)\neq \emptyset$.
Let $u\in N(w)\backslash N(v)$. Then $u\sim w$ and $u\not\sim v$.
If $G+\{u_{k}w,uv\}-\{u_{k}v,uw\}$ is connected, then $\{u_{k}v,uw\}$ is a bad pair of edges in $G$, which contradicts Corollary \ref{nobadedge}.
Suppose that $G+\{u_{k}w,uv\}-\{u_{k}v,uw\}$ is disconnected.
Thus, it is easy to see that $G-\{u_{k}v,uw\}$ contains exactly two components $H_{1}$ and $H_{2}$,
where $\{u_{k},w\}\subset H_{1}$ and $\{u,v\}\subset H_{2}$.
Let $\mathsf{P}$ be a path from $u$ to $v$ in $H_{2}$.
Choose a vertex $u'\in N(w)\backslash \{u\}$, we can see that $u'$ belongs to $H_{1}$.
Obviously, $u_{k}wu\mathsf{P}vu'$ forms a path in the graph $G+\{u_{k}w, u'v\}-\{u_{k}v,u'w\}$, and so this graph is connected.
Note also that $x_{u'}>x_{u_{k}}$ and $x_{v}\geq x_{w}$. Thus, $\{u_{k}v,u'w\}$ is a bad pair of edges in $G$, a contradiction.
Therefore, $x_{v}<\min\{x_{w}:w\in \{v_{k+1},\ldots,v_{s}\}\backslash\{v\}\}$, that is, $v=v_{k+1}$.
 Hence, $u_{k}\sim v_{k+1}$ and $x_{v_{k+1}}<x_{v_{k+2}}\leq \cdots\leq x_{v_{s}}$.
 A similar argument yields that $v_{k}\sim u_{k+1}$ and $x_{u_{k+1}}<x_{u_{k+2}}\leq \cdots\leq x_{u_{s}}$.
Thus, the proof of Claim 1 is completed.
\end{proof}

In summary, Claim 1 presents the adjacency relationships for all vertices in $\{u_{1},\ldots,u_{s-3},v_{1},\ldots,v_{s-3}\}$,
and shows that $u_{s-2}\sim v_{s-3}$ and $v_{s-2}\sim u_{s-3}$.
Then it is easy to see that the subgraph of $G$ induced by $\{u_{s-2},u_{s-1},u_{s},v_{s-2},v_{s-1},v_{s}\}$ must be isomorphic to $K_{3,3}-u_{s-2}v_{s-2}$.
Therefore, $G\cong B_{2s}$, as shown in Figure \ref{fig1}.

\medskip

\noindent\textbf{Case 2.} The degree sequence of $G$ is $(3,3,\ldots,3 | 3,3,\ldots,3,2,1)$.

\medskip

Assume that $G$ has the bipartition $(U,V)$, where $U=\{u_{1},u_{2},\ldots,u_{s}\}$ and $V=\{v_{0},v_{1},v_{2},\ldots,v_{s}\}$. Suppose that
$$x_{u_{1}}\leq x_{u_{2}}\leq\cdots\leq x_{u_{s}}~~~~\text{and}~~~~x_{v_{0}}\leq x_{v_{1}}\leq x_{v_{2}}\leq\cdots\leq x_{v_{s}}.$$
By Lemma \ref{eige-ord}, we can see that $d_{G}(v_{0})=1$, $d_{G}(v_{1})=2$ and $x_{v_{0}}<x_{v_{1}}< x_{v_{2}}\leq\cdots\leq x_{v_{s}}$.
Suppose that $N(v_{0})=u$. We claim that $x_{u}<\min\{x_{w}:w\in \{u_{1},\ldots,u_{s}\}\backslash\{u\}\}$.
If not, assume that $x_{u_{j}}\leq x_{u}$ where $u_{j}\in \{u_{1},\ldots,u_{s}\}\backslash\{u\}$.
Since $d_{G}(u)=d_{G}(u_{j})=3$, there is a vertex in $v\in N(u_{j})$ such that $u\not\sim v$.
According to Lemma \ref{lem-nocut}, one can see that $v_{0}u$ is the only one cut edge in $G$.
This implies that $G-v_{0}$ is 2-connected (otherwise, if there is a cut edge in $G-v_{0}$, then it is also a cut edge in $G$).
Let us consider the graph $G+\{v_{0}u_{j},vu\}-\{v_{0}u,vu_{j}\}$.
Since $G-v_{0}$ is 2-connected, the graph $G+\{v_{0}u_{j},vu\}-\{v_{0}u,vu_{j}\}$ is obviously connected.
Since $x_{v_{0}}<x_{v}$ and $x_{u_{j}}\leq x_{u}$,
$\{v_{0}u,vu_{j}\}$ is a bad pair of edges in $G$, a contradiction.
Therefore, we obtain that $u=u_{1}$ and $x_{u_{1}}<x_{u_{2}}\leq \cdots\leq x_{u_{s}}$.
By a simple inductive argument similar to the one used in Case 1, one can derive the adjacency relationships of the remaining vertices.
The graph $G$ is isomorphic to $B_{2s+1}$, which is displayed in Figure \ref{fig1}.
\end{proof}

\section{Spectral radius of $k$-connected irregular graphs}

Let us recall an inequality proposed by Shi \cite[Lemma 1]{Shi2009}. If $a,b>0$, then
\begin{equation}\label{eq4}
  a(p-q)^{2}+bq^{2}\geq\frac{abp^{2}}{a+b},
\end{equation}
with equality if and only if $q=ap/(a+b)$. Multiplying both side by $a+b$, we obtain that the above inequality is reduced to $[ap-(a+b)q]^{2}\geq 0$.
The inequality then follows directly.

Now, we are ready to present the proof of Theorem \ref{main1}.

\begin{proof}[\bf Proof of Theorem \ref{main1}]
Suppose that $x$ is a unit principal eigenvector of $G$.
  We may assume that $\hat{u}$ and $\hat{v}$ are two vertices in $G$ such that
  $x_{\hat{u}}=\max\{x_w:w\in V(G)\}$ and $x_{\hat{v}}=\min\{x_w:w\in V(G)\}$.
  If $d_{G}(\hat{u})<\Delta$, then
  $$\lambda_{1}(G)x_{\hat{u}}=\sum_{w\in N(u)}x_w\leq (\Delta-1)x_{\hat{u}},$$
  which leads to $\lambda_{1}(G)\leq \Delta-1$. Thus, the result holds since the right side of the inequality (\ref{eq0}) is less than one.

  In the following, we assume that $d_{G}(\hat{u})=\Delta$. Since $x$ is a principal eigenvector of $G$, we have
  \begin{eqnarray}\label{eq1}
    \lambda_{1}(G)=2\sum_{uv\in E(G)}x_{u}x_{v}=\sum_{uv\in E(G)}(x_{u}^{2}+x_{v}^{2})-\sum_{uv\in E(G)}(x_{u}-x_{v})^{2}
    =\sum_{w\in V(G)}d_{G}(w)x_{w}^{2}-\sum_{uv\in E(G)}(x_{u}-x_{v})^{2}.
  \end{eqnarray}
  Thus, it follows from (\ref{eq1}) that
  \begin{eqnarray}\label{eq2}
    \Delta-\lambda_{1}(G)&=&\Delta-\sum_{w\in V(G)}d_{G}(w)x_{w}^{2}+\sum_{uv\in E(G)}(x_{u}-x_{v})^{2}\nonumber\\
    &=&\sum_{w\in V(G)}(\Delta-d_{G}(w))x_{w}^{2}+\sum_{uv\in E(G)}(x_{u}-x_{v})^{2}\nonumber\\
    &\geq &(n\Delta-2m)x_{\hat{v}}^{2}+\sum_{uv\in E(G)}(x_{u}-x_{v})^{2}.
  \end{eqnarray}
One can see that $\sum_{uv\in E(G)}(x_{u}-x_{v})^{2}>0$ as $G$ is irregular. If $$x_{\hat{v}}^{2}\geq \frac{k^{2}}{(n\Delta-2m)\left((n-1)^{2}-(n-k-1)(\Delta-k+1)\right)+nk^{2}},$$
then the inequality (\ref{eq2}) and the fact $\sum_{uv\in E(G)}(x_{u}-x_{v})^{2}>0$ show that
$$\Delta-\lambda_{1}(G)>(n\Delta-2m)x_{\hat{v}}^{2}\geq \frac{(n\Delta-2m)k^{2}}{(n\Delta-2m)\left((n-1)^{2}-(n-k-1)(\Delta-k+1)\right)+nk^{2}},$$
and hence the result follows. Note also that $|N(\hat{v})|\geq k$. Let $w_{1},w_{2},\ldots,w_{k}$ be $k$ vertices in $N(\hat{v})$.
It is obvious that
  \begin{eqnarray}\label{eq3}
    \sum_{uv\in E(G)}(x_{u}-x_{v})^{2}\geq \sum_{i=1}^{k}(x_{w_{i}}-x_{\hat{v}})^{2}.
  \end{eqnarray}
Combining (\ref{eq2}) and (\ref{eq3}), we have
\begin{eqnarray}\label{eq5}
    \Delta-\lambda_{1}(G)&\geq& (n\Delta-2m)x_{\hat{v}}^{2}+\sum_{i=1}^{k}(x_{w_{i}}-x_{\hat{v}})^{2} \nonumber\\
    &=&\sum_{i=1}^{k}\left(\frac{n\Delta-2m}{k}x_{\hat{v}}^2+(x_{w_i}-x_{\hat{v}})^2\right)\nonumber\\
    &\geq& \sum_{i=1}^{k}\frac{n\Delta-2m}{n\Delta-2m+k}x_{w_{i}}^2,
  \end{eqnarray}
  where the last inequality follows from (\ref{eq4}) with $a=1$, $b=\frac{n\Delta-2m}{k}$, $p=x_{w_{i}}$ and $q=x_{\hat{v}}$.
If
$$\sum_{i=1}^{k}x_{w_{i}}^2>\frac{(n\Delta-2m)k^{2}+k^{3}}{(n\Delta-2m)\left((n-1)^{2}-(n-k-1)(\Delta-k+1)\right)+nk^{2}},$$
then the result follows from (\ref{eq5}). So, in the following, we may assume that
$$x_{\hat{v}}^{2}<\frac{k^{2}}{(n\Delta-2m)\left((n-1)^{2}-(n-k-1)(\Delta-k+1)\right)+nk^{2}}$$
and
$$\sum_{i=1}^{k}x_{w_{i}}^2\leq \frac{(n\Delta-2m)k^{2}+k^{3}}{(n\Delta-2m)\left((n-1)^{2}-(n-k-1)(\Delta-k+1)\right)+nk^{2}}.$$
Since $k\leq n-2$, the set $V(G)\setminus\{\hat{v},w_{1},w_{2},\ldots,w_{k}\}$ is nonempty.
   Thus, we have
   \begin{equation}\label{eq6}
   x_{\hat{u}}^{2}\geq \frac{1-x_{\hat{v}}^{2}-\sum_{i=1}^{k}x_{w_{k}}^{2}}{n-k-1}.
   \end{equation}
   By (\ref{eq6}), we obtain that
\begin{equation}\label{eq9}
  x_{\hat{u}}^{2}>\left(1-\frac{(n\Delta-2m+1)k^{2}+k^{3}}{(n\Delta-2m)\left((n-1)^{2}-(n-k-1)(\Delta-k+1)\right)+nk^{2}}\right)/(n-k-1).
\end{equation}
 Since $G$ is $k$-connected, Menger's theorem shows that there are $k$ internally vertex disjoint paths from $\hat{u}$ to $\hat{v}$.
  Let $P_{1},P_{2},\ldots,P_{k}$ be internally vertex disjoint $(\hat{u},\hat{v})$-paths,
  such that $|N(\hat{u})\cap V(P_{i})|=1$ for $1\leq i\leq k$.
  Apart from these paths, there are at least $\Delta-k$ vertices.
  Thus, $$|V(P_{1}\cup P_{2}\cup \cdots\cup P_{k})|+\Delta-k\leq n.$$
  Note that $\sum_{i=1}^{k}|V(P_{i})|=|V(P_{1}\cup P_{2}\cup \cdots\cup P_{k})|+2(k-1)$. It follows that
  $$\sum_{i=1}^{k}|V(P_{i})|\leq n+3k-\Delta-2.$$
  Therefore, it follows from Cauchy-Schwarz inequality that
  \begin{eqnarray}\label{eq7}
    \sum_{uv\in E(G)}(x_{u}-x_{v})^{2}&\geq& \sum_{i=1}^{k}\sum_{uv\in E(P_{i})}(x_{u}-x_{v})^{2}\nonumber\\
     &\geq& \frac{1}{\sum_{i=1}^{k}(|V(P_{i})|-1)}\left(\sum_{i=1}^{k}\sum_{uv\in E(P_{i})}(x_{u}-x_{v})\right)^{2}\nonumber\\
     &=&\frac{k^2}{\sum_{i=1}^k(|V(P_t)|-1)}(x_{\hat{u}}-x_{\hat{v}})^2\nonumber\\
     &\geq& \frac{k^2}{n-\Delta+2k-2}(x_{\hat{u}}-x_{\hat{v}})^2.
  \end{eqnarray}
Combining (\ref{eq2}) and (\ref{eq7}), it follows that
\begin{eqnarray}\label{eq8}
  \Delta-\lambda_{1}(G) &\geq&(n\Delta-2m)x_{\hat{v}}^{2}+\frac{k^2}{n-\Delta+2k-2}(x_{\hat{u}}-x_{\hat{v}})^2\nonumber\\
  &\geq& \frac{k^2(n\Delta-2m)}{(n\Delta-2m)(n-\Delta+2k-2)+k^{2}}x_{\hat{u}}^{2},
\end{eqnarray}
 where the last inequality follows from (\ref{eq4}) with $a=\frac{k^{2}}{n-\Delta+2k-2}$, $b=n\Delta-2m$, $p=x_{\hat{u}}$ and $q=x_{\hat{v}}$.
 Replacing (\ref{eq9}) in (\ref{eq8}), we have
$$\Delta-\lambda_{1}(G)>\frac{(n\Delta-2m)k^{2}}{(n\Delta-2m)\left((n-1)^{2}-(n-k-1)(\Delta-k+1)\right)+nk^{2}},$$
as required.
\end{proof}

\noindent\textbf{Remark.} Let us compare the previous bounds and our bound obtained in Theorem \ref{main1}. Take $k=1$ in Theorem \ref{main1},
the lower bound (\ref{eq0}) is equivalent to
$$\Delta-\lambda_{1}(G)>\frac{1}{(n-1)^{2}-(n-2)\Delta+\frac{n}{n\Delta-2m}}.$$
Since $n\Delta-2m\geq 1$ and $\Delta>1$,
we have
$$(n-1)^{2}-(n-2)\Delta+\frac{n}{n\Delta-2m}<(n-1)^{2}-(n-2)+n<2n(n-1)<2n(n\Delta-1)\Delta^{2},$$
which implies that the lower bound (\ref{eq0}) is better than the bound (\ref{eq18}) in Theorem \ref{th1}.
Moreover, since $$(n-1)^{2}-(n-k-1)(\Delta-k+1)<n^{2}-2n+2k,$$ the lower bound (\ref{eq0}) is also better than the bound (\ref{eq19}) in Theorem \ref{th2}.
Generally, the lower bounds (\ref{eq0}) and (\ref{eq22}) are incomparable.
All trees with 9 vertices were exhibited in \cite[Appendix Table A4]{Cvetkovic2010}.
For nine of these trees, with serial numbers $\{53, 59, 61-64, 67, 72, 74\}$, the lower bound (\ref{eq22}) is better than (\ref{eq0}).
For other trees on 9 vertices, the lower bound (\ref{eq0}) is better.

\medskip

At the end of this section, we will discuss the spectral radius of the proper subgraph of regular graphs.
Note that the complete graph $K_{n}$ is a regular graph, and its spectral radius is $n-1$.
Let $K_{n}^{-}$ be the graph obtained from the complete graph $K_{n}$ by deleting an edge.
One can easily see that the spectral radius of $K_{n}^{-}$ is
$$\frac{n-3+\sqrt{n^{2}+2n-7}}{2}.$$

In order to prove Theorem \ref{main2}, we consider the spectral radius of nearly regular graphs.

\begin{lemma}\label{lem1}
  Let $G$ be a $k$-connected $\Delta$-regular graph of order $n$. If $H$ is a graph obtained from $G$ by deleting an edge, then
  \begin{equation}\label{eq13}
    \Delta-\lambda_{1}(H)>\frac{k^{2}}{(n-\Delta-1)(n-\Delta+2k-2)+nk^{2}}.
    \end{equation}
\end{lemma}

\begin{proof}
  Suppose that $H\cong G-\hat{u}\hat{v}$ where $\hat{u}\hat{v}$ is an edge of $G$.
  Thus, $d_{H}(\hat{u})=d_{H}(\hat{v})=\Delta-1$, and $d_{H}(w)=\Delta$ for $w\in V(H)\backslash\{\hat{u},\hat{v}\}$.
  If $\Delta=n-1$, then $G\cong K_{n}$ and $H\cong K_{n}^{-}$.
  Clearly, $$\Delta-\lambda_{1}(K_{n}^{-})=n-1-\frac{n-3+\sqrt{n^{2}+2n-7}}{2}>\frac{1}{n},$$ and the result follows.
  Assume that $\Delta\leq n-2$.
  Let $x$ be a unit principal eigenvector of $H$. As mentioned in (\ref{eq1}),
  \begin{equation}
    \lambda_{1}(H)=\sum_{w\in V(G)}d_{H}(w)x_{w}^{2}-\sum_{uv\in E(G)}(x_{u}-x_{v})^{2}.
  \end{equation}
  Therefore, we obtain that
  \begin{equation}\label{eq11}
    \Delta-\lambda_{1}(H)=\Delta\sum_{w\in V(G)}x_{w}^{2}-\sum_{w\in V(G)}d_{H}(w)x_{w}^{2}+\sum_{uv\in E(G)}(x_{u}-x_{v})^{2}
    =x_{\hat{u}}^{2}+x_{\hat{v}}^{2}+\sum_{uv\in E(G)}(x_{u}-x_{v})^{2}.
  \end{equation}
  Suppose that $N_{H}(\hat{u})=\{u_{1},u_{2},\ldots,u_{\Delta-1}\}$. By (\ref{eq4}) and (\ref{eq11}), we have
\begin{eqnarray}\label{eq14}
  \Delta-\lambda_{1}(H)\geq x_{\hat{u}}^{2}+\sum_{i=1}^{\Delta-1}(x_{u_{i}}-x_{\hat{u}})^{2}
  =\sum_{i=1}^{\Delta-1}\left(\frac{1}{\Delta-1}x_{\hat{u}}^{2}+(x_{u_{i}}-x_{\hat{u}})^{2}\right)
  \geq \frac{1}{\Delta}\sum_{i=1}^{\Delta-1}x_{u_{i}}^{2}.
\end{eqnarray}
Since $\sum_{uv\in E(G)}(x_{u}-x_{v})^{2}>0$, by (\ref{eq11}), the lemma holds if
$$x_{\hat{u}}^{2}+x_{\hat{v}}^{2}\geq \frac{k^{2}}{(n-\Delta-1)(n-\Delta+2k-2)+nk^{2}}.$$
Moreover, if
$$\sum_{i=1}^{\Delta-1}x_{u_{i}}^{2}>\frac{\Delta k^{2}}{(n-\Delta-1)(n-\Delta+2k-2)+nk^{2}},$$
the result follows immediately from (\ref{eq14}). So, in the following, we may assume that
\begin{equation}\label{eq15}
  x_{\hat{u}}^{2}+x_{\hat{v}}^{2}<\frac{k^{2}}{(n-\Delta-1)(n-\Delta+2k-2)+nk^{2}}
\end{equation}
and
\begin{equation}\label{eq16}
  \sum_{i=1}^{\Delta-1}x_{u_{i}}^{2}\leq \frac{\Delta k^{2}}{(n-\Delta-1)(n-\Delta+2k-2)+nk^{2}}.
\end{equation}
  Note that the average degree of $H$ is $\Delta-\frac{2}{n}$, and so $\lambda_{1}(H)>\Delta-\frac{2}{n}$.
  Suppose that $x_{\hat{u}}\geq x_{\hat{v}}$. It is easy to see that $x_{\hat{u}}>0,$ whether $H$ is connected or not.
  If $x_{\hat{u}}=\max\{x_{w}:w\in V(H)\}$,
  then $$\lambda_{1}(H)x_{\hat{u}}=\sum_{w\hat{u}\in E(G)}x_{w}\leq (\Delta-1)x_{\hat{u}},$$
  which implies that $\lambda_{1}(H)\leq \Delta-1$. But this contradicts the fact $\lambda_{1}(H)>\Delta-2/n$.
  Therefore, we obtain that $\max\{x_{w}:w\in V(H)\}>x_{\hat{u}}\geq x_{\hat{v}}$.
  Assume that $\hat{w}$ is a vertex in $H$ such that $x_{\hat{w}}=\max\{x_{w}:w\in V(H)\}$.
  Since $\Delta\leq n-2$, it follows from the assumptions (\ref{eq15}) and (\ref{eq16}) that
  \begin{equation}\label{eq17}
    x_{\hat{w}}\geq \frac{1-x_{\hat{u}}^{2}-x_{\hat{v}}^{2}-\sum_{i=1}^{\Delta-1}x_{u_{i}}^{2}}{n-\Delta-1}
    >\frac{n-\Delta+k^{2}+2k-2}{(n-\Delta-1)(n-\Delta+2k-2)+nk^{2}}.
  \end{equation}
  We claim that the eigenvector $x$ satisfies
  \begin{equation}\label{eq10}
    \sum_{uv\in E(H)}(x_{u}-x_{v})^{2}\geq \frac{k^{2}}{n+2k-\Delta-2}(x_{\hat{w}}-x_{\hat{u}})^{2}.
  \end{equation}
  Since $G$ is $k$-connected, Menger's Theorem shows that there are $k$ internally vertex disjoint paths from $\hat{w}$ to $\hat{u}$.
  Suppose that $P_{1},P_{2},\ldots,P_{k}$ are internally vertex disjoint $(\hat{w},\hat{u})$-paths in $G$,
  such that $|N_{G}(\hat{w})\cap V(P_{i})|=1$ for $1\leq i\leq k$.
  If $\hat{v}\notin \bigcup_{i=1}^{k}V(P_{i})$, then $P_{1},\ldots,P_{k}$ are also paths in $H$.
  As mentioned in the proof of Theorem \ref{main1}, we have
  $$\sum_{i=1}^{k}|V(P_{i})|\leq n+3k-\Delta-2.$$
  Therefore, it follows from Cauchy-Schwarz inequality that
  \begin{eqnarray*}
    \sum_{uv\in E(H)}(x_{u}-x_{v})^{2}&\geq& \sum_{i=1}^{k}\sum_{uv\in P_{i}}(x_{u}-x_{v})^{2}\\
     &\geq& \frac{1}{\sum_{i=1}^{k}(|V(P_{i})|-1)}\left(\sum_{i=1}^{k}\sum_{uv\in P_{i}}(x_{u}-x_{v})\right)^{2}\\
     &=&\frac{k^2}{\sum_{i=1}^k(|V(P_t)|-1)}(x_{\hat{w}}-x_{\hat{u}})^2\\
     &\geq& \frac{k^2}{n-\Delta+2k-2}(x_{\hat{w}}-x_{\hat{u}})^2,
  \end{eqnarray*}
  as claimed. Suppose that $\hat{v}\in V(P_{k})$ and $P_{k}=\hat{w}\cdots \hat{v}\hat{u}$.
   Let $P_{k}^{*}$ be the path obtained from $P_{k}$ by deleting the vertex $\hat{u}$, that is, $P_{k}^{*}=\hat{w}\cdots \hat{v}$.
   Clearly, $P_{1},\ldots,P_{k-1},P_{k}^{*}$ are also paths in $H$. It follows that
   $$|V(P_{k}^{*})|+\sum_{i=1}^{k-1}|V(P_{i})|\leq n+3k-\Delta-3<n+3k-\Delta-2.$$
   A similar argument, using Cauchy-Schwarz inequality, yields that
   \begin{eqnarray*}
    \sum_{uv\in E(H)}(x_{u}-x_{v})^{2}&\geq& \frac{1}{|V(P_{k}^{*})|-1+\sum_{i=1}^{k-1}(|V(P_{i})|-1)}\left(\sum_{uv\in P_{k}^{*}}(x_{u}-x_{v})+\sum_{i=1}^{k-1}\sum_{uv\in P_{i}}(x_{u}-x_{v})\right)^{2}\\
    &>&\frac{1}{n-\Delta+2k-2}\left(x_{\hat{w}}-x_{\hat{v}}+(k-1)(x_{\hat{w}}-x_{\hat{u}})\right)^{2}\\
     &\geq& \frac{k^2}{n-\Delta+2k-2}(x_{\hat{w}}-x_{\hat{u}})^2,
  \end{eqnarray*}
   where the last inequality holds as $x_{\hat{u}}\geq x_{\hat{v}}$. Therefore, the eigenvector $x$ satisfies the inequality (\ref{eq10}).
    Combining (\ref{eq4}), (\ref{eq11}) and (\ref{eq10}), we obtain that
    \begin{equation}\label{eq12}
    \Delta-\lambda_{1}(H)\geq x_{\hat{u}}^{2}+\frac{k^2}{n-\Delta+2k-2}(x_{\hat{w}}-x_{\hat{u}})^2
    \geq\frac{k^2}{n-\Delta+k^{2}+2k-2}x_{\hat{w}}^2.
  \end{equation}
Thus, the result follows directly by using (\ref{eq17}) in the above inequality, and the proof is completed.
\end{proof}

Theorem \ref{main2} is now a direct consequence of the above lemma.

\begin{proof}[\bf Proof of Theorem \ref{main2}]
Since $H$ is a proper subgraph of $G$, there exists a graph $H'$ satisfying the following
properties: (a) $H'$ is obtained from $G$ by deleting an edge; (b) $H$ is a subgraph of $H'$ ($H\cong H'$ is allowed).
According to the property (b), we obtain $\lambda_{1}(H)\leq \lambda_{1}(H')$. Combining the property (a) and Lemma \ref{lem1},
it follows that
$$\Delta-\lambda_{1}(H)\geq \Delta-\lambda_{1}(H')>\frac{k^{2}}{(n-\Delta-1)(n-\Delta+2k-2)+nk^{2}},$$
which completes the proof.
\end{proof}

\noindent\textbf{Remark.} In Theorem \ref{th3}, the supergraph $G$ is $k$-connected, where $k\geq 2$.
The connectivity restriction is relaxed to any positive integer $k$ in our result.
Moreover, we improve the bound in Theorem \ref{th3} when $k\geq 2$.
Set $\Phi_{1}=(k-1)^{2}(n-\Delta-1)(n-\Delta+2k-2)$ and $\Phi_{2}=k^{2}(n-\Delta)(n-\Delta+2k-4)$.
One can see that
$$\Phi_{2}-\Phi_{1}=(2k-1)(n-\Delta)^{2}+(3k^{2}-8k+3)(n-\Delta)+2(k-1)^{3}>0,$$
as $k\geq 2$ and $n-\Delta\geq 1$.
Indeed, the fact $\Phi_{2}>\Phi_{1}$ indicates that the lower bound (\ref{eq20}) in our result is better than the bound (\ref{eq21}) in Theorem \ref{th3}.

\section{Discussion}

The unique maximal connected subcubic bipartite graph is determined in the paper. As shown in Figure \ref{fig1}, the maximal graph $B_{n}$ is path-like.
Very recently, the asymptotic value for the spectral radius of $B_{n}$ is provided in \cite{Xue2022}, that is,
$$\lim_{n \to \infty}n^{2}(3-\lambda_{1}(B_{n}))=\pi^{2}.$$
Based on this approach, one problem arises: determine the maximal graph in $\mathcal{B}(n,\Delta)$ for $\Delta\geq 4$.
According to our observations, in order to determine the maximal graph, one often requires to find possible degree sequences.
If the above problem can not be solved, one can deal with possible degree sequences of the maximal graph in $\mathcal{B}(n,\Delta)$ for $\Delta\geq 4$.

\section*{Acknowledgements}
We are greatly appreciated to Professors Dragan Stevanovi\'{c} and Sebastian M. Cioab\u{a} for valuable suggestions and comments,
which are very helpful to improve the presentation of our paper.
This work was supported by National Natural Science Foundation of China (Nos. 12001498 and 11971445),
China Postdoctoral Science Foundation (No. 2022TQ0303) and Natural Science Foundation of Henan Province (No. 202300410377).

\end{document}